\documentclass{article}

\usepackage[letterpaper, margin=1in]{geometry}
\usepackage{amsmath,amssymb}
\usepackage{enumerate}
\usepackage{booktabs}
\usepackage[amsmath,thmmarks,framed]{ntheorem}
\usepackage{authblk}
\usepackage{xcolor}
\definecolor{darkgreen}{rgb}{0.0,0.5,0.0}
\usepackage[colorlinks,linkcolor=red,citecolor=darkgreen,urlcolor=blue]{hyperref}
\hypersetup{
  pdftitle={Classification of the orbits of three quadratic Collatz-type maps},
  pdfauthor={Xuda Ye}
}
\usepackage[numbers,square,comma,sort&compress]{natbib}

\newtheorem{theorem}{Theorem}
\newtheorem{lemma}{Lemma}
\newtheorem{corollary}{Corollary}
\newtheorem{remark}{Remark}
\newenvironment{proof}[1][Proof]{\par\medskip\noindent\textit{#1.}\ }{\hfill$\square$\par\medskip}

\newcommand{\Le}{\leqslant}
\newcommand{\Ge}{\geqslant}
\newcommand{\ord}{\operatorname{ord}}

\title{Classification of the Orbits of Three Quadratic Collatz-type Maps}
\author{Xuda Ye}
\affil{Department of Mathematics, Purdue University, West Lafayette, IN 47907}
\date{\today}

\begin{document}
\maketitle

\begin{abstract}
We study three quadratic Collatz-type maps on the non-negative integers. Each of them halves an even number, and they send an odd $n$ to $n(n-1)/2$, to $n(n+1)/2$ and to $(n^2 - 1)/4$, respectively.  Along an orbit, the odd terms grow until they reach a number of the form $2^m + 1$ or $2^m - 1$, depending on the map, and an orbit with an odd term greater than $1$ is bounded exactly when this happens. The step that reaches such a number gives a solution of an exponential Diophantine equation, which becomes the equation $2^a \pm 2^b + 1 = z^2$ after completing the square. With Szalay's classification of the solutions of this equation, we classify all orbits of the maps. Every orbit either is eventually periodic or tends to infinity, and we determine exactly the initial values with bounded orbits. For the map with odd branch $n(n-1)/2$, Sedaghat stated the classification, and we give a new and complete proof. For this map, we also give an elementary criterion for unbounded orbits, based on divisibility alone. To our knowledge, the classifications for the other two maps are new.
\end{abstract}



\section{Introduction}
\label{sec: intro}

On the set $\mathbb N_0 = \{0, 1, 2, \dots\}$ of non-negative integers, consider the maps
\begin{equation}
    Q_-(n) =
    \begin{cases}
        \dfrac n2 , & n \text{ even} , \\[6pt]
        \dfrac{n(n-1)}{2} , & n \text{ odd} ,
    \end{cases}
    \qquad
    Q_+(n) =
    \begin{cases}
        \dfrac n2 , & n \text{ even} , \\[6pt]
        \dfrac{n(n+1)}{2} , & n \text{ odd} ,
    \end{cases}
    \label{eq: maps}
\end{equation}
whose odd branches are the binomial coefficients $\binom n2$ and $\binom{n+1}2$, and the symmetric map
\begin{equation}
    S(n) =
    \begin{cases}
        \dfrac n2 , & n \text{ even} , \\[6pt]
        \dfrac{n^2 - 1}{4} , & n \text{ odd} .
    \end{cases}
    \label{eq: map S}
\end{equation}
The orbit $(x_i)_{i \Ge 0}$ of an initial value $x_0 \in \mathbb N_0$ under $Q_-$ is defined by $x_{i+1} = Q_-(x_i)$, so that $x_i = Q_-^i(x_0)$, where $Q_-^i$ denotes the $i$-fold composition of $Q_-$. The orbits under $Q_+$ and $S$ and the compositions $Q_+^i$ and $S^i$ are defined in the same way. For $x_0 \Ge 1$, we write $x_0 = 2^l q$ with $l \Ge 0$ and $q$ odd, and call $q$ the odd part of $x_0$. In its compressed form, the Collatz map sends an even $n$ to $n/2$ and an odd $n$ to $(3n+1)/2$. As in the study of the Collatz map, we ask which initial values have bounded orbits, where these orbits end, and whether all other orbits tend to infinity.

Our main results are the following theorems. They characterize all orbits of the Collatz-type maps $Q_-$, $Q_+$ and $S$ in \eqref{eq: maps} and \eqref{eq: map S}.
\clearpage

\begin{theorem}
\label{thm: minus}
    Let $x_0 \in \mathbb N_0$, and let $(x_i)_{i \Ge 0}$ be its orbit under $Q_-$.
    \begin{enumerate}[(i)]
        \setlength{\itemsep}{0pt}
        \item If $x_0 = 0$ or $x_0 = 2^l$ with $l \Ge 0$, the orbit reaches the fixed point $0$.
        \item If $x_0 = 2^l (2^m + 1)$ with $l \Ge 0$ and $m \Ge 1$, the orbit enters the cycle of length $m$ that contains $2^m + 1$.
        \item For every other $x_0$, the orbit satisfies $x_i \to \infty$ as $i \to \infty$.
    \end{enumerate}
\end{theorem}

\begin{theorem}
\label{thm: plus}
    Let $x_0 \in \mathbb N_0$, and let $(x_i)_{i \Ge 0}$ be its orbit under $Q_+$.
    \begin{enumerate}[(i)]
        \setlength{\itemsep}{0pt}
        \item If $x_0 = 0$, the orbit stays at the fixed point $0$.
        \item If $x_0 = 2^l (2^m - 1)$ with $l \Ge 0$ and $m \Ge 1$, the orbit enters the cycle of length $m$ that contains $2^m - 1$, which for $m = 1$ is the fixed point $1$.
        \item If $x_0 = 5 \cdot 2^l$ with $l \Ge 0$, the orbit enters the cycle $15 \to 120 \to 60 \to 30 \to 15$ through $15 = 2^4 - 1$.
        \item For every other $x_0$, the orbit satisfies $x_i \to \infty$ as $i \to \infty$.
    \end{enumerate}
\end{theorem}

\begin{theorem}
\label{thm: S}
    Let $x_0 \in \mathbb N_0$, and let $(x_i)_{i \Ge 0}$ be its orbit under $S$.
    \begin{enumerate}[(i)]
        \setlength{\itemsep}{0pt}
        \item If $x_0 = 0$, or if the odd part of $x_0$ is of the form $2^m \pm 1$ with $m \Ge 1$ or lies in $\{11, 23, 181\}$, the orbit reaches the fixed point $0$.
        \item For every other $x_0$, the orbit satisfies $x_i \to \infty$ as $i \to \infty$.
    \end{enumerate}
\end{theorem}

\citet{sedaghat2020bounded} studies $Q_-$ as a quadratic analogue of the Collatz-type map that sends an even $n$ to $n/2$ and an odd $n$ to $(3n-1)/2 = n + (n-1)/2$, and calls $Q_-$ the divide-or-choose-2 rule. In $Q_-$ and $Q_+$, the sums $n + (n-1)/2$ and $(3n+1)/2 = n + (n+1)/2$ are replaced by the products $n \cdot (n-1)/2$ and $n \cdot (n+1)/2$. \citet[Section~1]{sedaghat2020bounded} states that the orbits of $Q_+$ are qualitatively similar to those of $Q_-$, and does not study $Q_+$ further. The odd branch of $S$ is the product of the consecutive integers $(n-1)/2$ and $(n+1)/2$, and \citet[Section~1]{sedaghat2020bounded} calls $S$ the symmetric rule. For $Q_-$, Sedaghat conjectures that unbounded orbits exist, and that every orbit containing an odd number $2^m - 1$ with $m \Ge 3$ is unbounded \citep[Conjectures~3 and~4]{sedaghat2020bounded}. For $S$, the existence of unbounded orbits is left open in \citep[Section~4]{sedaghat2020bounded}.

\citet[Theorem~3]{sedaghat2020dual} states that the orbit of $2^l (2^m + 1)$ with $l, m \Ge 1$ under $Q_-$ ends in a cycle of length $m$, and that every other orbit goes to infinity. The proof there reduces the classification to the statement that $k(2^j k + 1) = 2^m + 1$ has no solution in integers $j, m \Ge 1$ and odd $k \Ge 3$ \citep[Lemma~2]{sedaghat2020dual}. This statement is proved there by a descent on powers of $2$. The proof does not treat the case in which the descent stops (Remark~\ref{rem: descent}). We are not aware of other work on $Q_+$ or on the classification of the orbits of $S$.

Our approach starts from the terminal numbers defined below, which differ between $Q_-$, $Q_+$ and $S$ and decide whether an orbit with an odd term greater than $1$ enters a cycle. Let $y > 1$ be an odd term of an orbit. We write
\begin{equation}
    y =
    \begin{cases}
        2^j k + 1 , & \text{under } Q_- , \\
        2^j k - 1 , & \text{under } Q_+ , \\
        2^{j+1} k + \varepsilon , & \text{under } S ,
    \end{cases}
    \qquad j \Ge 1 , \quad k \text{ odd} ,
    \label{eq: decomposition}
\end{equation}
where $\varepsilon \in \{-1, 1\}$ is chosen under $S$ such that $4$ divides $y - \varepsilon$. We also set $\varepsilon = 1$ under $Q_-$ and $\varepsilon = -1$ under $Q_+$. For each odd $y > 1$, the numbers $j$, $k$ and $\varepsilon$ in the decomposition \eqref{eq: decomposition} are unique.

In each case, the next odd term after $y$ in the orbit is $y' = k (2^j k + \varepsilon)$ (Lemma~\ref{lem: odd step}). If $k \Ge 3$, then $y' > y$. If $k = 1$, then $y$ is a number $2^m + 1$ under $Q_-$, a number $2^m - 1$ under $Q_+$, and a number $2^m \pm 1$ under $S$. We call these numbers, with $m \Ge 1$, the \emph{terminal numbers} of the map. They lie on cycles under $Q_-$ and $Q_+$, and their orbits reach $0$ under $S$. Hence, if the odd part of $x_0$ exceeds $1$, the orbit is bounded if and only if its odd terms reach a terminal number (Lemma~\ref{lem: orbit}). If the odd part of $x_0$ is not a terminal number but a later odd term is, the earliest step that reaches a terminal number $2^m + \delta$ gives a solution of
\begin{equation}
	\boxed{
    k \bigl( 2^j k + \varepsilon \bigr) = 2^m + \delta , \qquad j, m \Ge 1 , \quad k \Ge 3 \text{ odd} , \quad \varepsilon, \delta \in \{-1, 1\} ,}
    \label{eq: key}
\end{equation}
with $\varepsilon = \delta = 1$ under $Q_-$, with $\varepsilon = \delta = -1$ under $Q_+$, and with $\varepsilon, \delta \in \{-1, 1\}$ under $S$ (Lemma~\ref{lem: reduction}). Multiplying \eqref{eq: key} by $2^{j+2}$ and adding $1$ gives
\begin{equation}
    \bigl( 2^{j+1} k + \varepsilon \bigr)^2 = 2^{m+j+2} + \delta \, 2^{j+2} + 1 ,
    \label{eq: square}
\end{equation}
so \eqref{eq: key} asks when a number $2^a \pm 2^b + 1$ is a square. This reduces the study of terminal numbers for the Collatz-type maps to a Diophantine equation.

\citet{szalay2002equations} determined all squares of the form $2^a \pm 2^b + 1$ with $a > b \Ge 1$. By this result, the solutions of \eqref{eq: key} are
\begin{equation}
    3 \cdot (2 \cdot 3 - 1) = 2^4 - 1 , \qquad 3 \cdot (4 \cdot 3 - 1) = 2^5 + 1 , \qquad 45 \cdot (2 \cdot 45 + 1) = 2^{12} - 1
    \label{eq: solutions}
\end{equation}
(Lemma~\ref{lem: key}). None of them has $\varepsilon = \delta = 1$, and only $3 \cdot (2 \cdot 3 - 1) = 2^4 - 1$ has $\varepsilon = \delta = -1$. We call an odd part $q > 1$ \emph{exceptional} if it is not a terminal number but its orbit is bounded. By the decomposition \eqref{eq: decomposition}, the solutions \eqref{eq: solutions} give the exceptional odd part $5 = 2 \cdot 3 - 1$ of $Q_+$ and the exceptional odd parts $11 = 4 \cdot 3 - 1$, $23 = 8 \cdot 3 - 1$ and $181 = 4 \cdot 45 + 1$ of $S$ (Table~\ref{tab: maps}).

\begin{table}[!ht]
\centering
\begin{tabular}{cccccc}
\toprule
map & odd branch & terminal numbers & equation \eqref{eq: key} & exceptional odd parts & theorem \\
\midrule
$Q_-$ & $n(n-1)/2$ & $2^m + 1$ & $k (2^j k + 1) = 2^m + 1$ & none & Theorem~\ref{thm: minus} \\
$Q_+$ & $n(n+1)/2$ & $2^m - 1$ & $k (2^j k - 1) = 2^m - 1$ & $5$ & Theorem~\ref{thm: plus} \\
$S$ & $(n^2 - 1)/4$ & $2^m \pm 1$ & $k (2^j k + \varepsilon) = 2^m + \delta$ & $11$, $23$, $181$ & Theorem~\ref{thm: S} \\
\bottomrule
\end{tabular}
\caption{The Collatz-type maps of \eqref{eq: maps} and \eqref{eq: map S}: the odd branch, the terminal numbers with $m \Ge 1$, the form of \eqref{eq: key} that arises in Lemma~\ref{lem: reduction}, the odd parts $q > 1$ that are not terminal numbers but have a bounded orbit, and the classification theorem.}
\label{tab: maps}
\end{table}

Theorem~\ref{thm: minus} is the classification of \citet[Theorem~3]{sedaghat2020dual}, with $l \Ge 0$ and with the initial values $0$ and $2^l$. We provide a new and complete proof: we transform $k(2^j k + 1) = 2^m + 1$ into the equation $2^a + 2^b + 1 = z^2$, whose solutions were determined by \citet{szalay2002equations}, and conclude that $k(2^j k + 1) = 2^m + 1$ has no solution with odd $k \Ge 3$ (Lemma~\ref{lem: key}~(i)). To our knowledge, Theorems~\ref{thm: plus} and~\ref{thm: S} are new. Theorem~\ref{thm: S} shows that $S$ has unbounded orbits, for instance the orbit of $13$. This answers the question left open by \citet{sedaghat2020bounded}. For each map, only $O((\log X)^2)$ initial values in $\{0, 1, \dots, X\}$ have a bounded orbit (Corollary~\ref{cor: bounded}).

According to \citet[Section~1]{scott2006elementary}, neither Szalay's proof for $2^a + 2^b + 1 = z^2$ nor the shorter proof of Scott is strictly elementary, so our proof of Theorem~\ref{thm: minus} is not strictly elementary. For $Q_-$, Theorem~\ref{thm: criterion} gives an elementary criterion: if an odd term $y > 1$ of an orbit divides none of the numbers $2^m + 1$, then the orbit tends to infinity. Moreover, $y$ divides such a number if and only if the multiplicative orders of $2$ modulo the prime factors of $y$ are even and have the same $2$-adic valuation. The criterion gives an elementary proof of both conjectures of \citet{sedaghat2020bounded}.

Section~\ref{sec: odd terms} studies the odd terms of an orbit under $Q_-$, $Q_+$ and $S$, and reduces the classification to \eqref{eq: key} (Lemma~\ref{lem: reduction}). Section~\ref{sec: equation} solves \eqref{eq: key} and discusses the descent of \citep{sedaghat2020dual}. Section~\ref{sec: proofs} proves Theorems~\ref{thm: minus}, \ref{thm: plus} and~\ref{thm: S} and Corollary~\ref{cor: bounded}. Section~\ref{sec: criterion} proves the divisibility criterion for $Q_-$ and uses it to show, without Szalay's result, that the initial values with a bounded orbit under $Q_-$ have natural density $0$. It also explains why the criterion does not carry over to $Q_+$ and $S$, and why Theorem~\ref{thm: plus} has an elementary proof. Section~\ref{sec: conclusion} states open questions.

\section{Odd terms of an orbit}
\label{sec: odd terms}

Throughout this section, an odd number $y \Ge 3$ is written as in \eqref{eq: decomposition}, with the sign $\varepsilon$ defined there.

\begin{lemma}
\label{lem: odd step}
    Let $y \Ge 3$ be odd, and write $y$ as in the decomposition \eqref{eq: decomposition}.
    \begin{enumerate}[(i)]
        \setlength{\itemsep}{0pt}
        \item $Q_-(y) = 2^{j-1} k y$ and $Q_-^j(y) = k (2^j k + 1)$.
        \item $Q_+(y) = 2^{j-1} k y$ and $Q_+^j(y) = k (2^j k - 1)$.
        \item $S(y) = 2^j k (2^j k + \varepsilon)$ and $S^{j+1}(y) = k (2^j k + \varepsilon)$.
    \end{enumerate}
    That is, in each case the next odd term after $y$ in the orbit is $y' = k (2^j k + \varepsilon)$, and the terms between $y$ and $y'$ are multiples of $y'$. If $k \Ge 3$, then $y' > y$.
\end{lemma}

\begin{proof}
    By \eqref{eq: maps} and the decomposition \eqref{eq: decomposition}, $Q_-(y) = y \cdot (y - 1)/2$ with $y - 1 = 2^j k$, and $Q_+(y) = y \cdot (y + 1)/2$ with $y + 1 = 2^j k$. In both cases, this is $2^{j-1} k y$. The $j - 1$ halvings that follow take $2^{j-1} k y$ to the odd number $y' = k y = k (2^j k + \varepsilon)$ through multiples of $y'$. If $k \Ge 3$, then $y' = k y > y$.

    Under $S$, we have $y + \varepsilon = 2^{j+1} k + 2 \varepsilon = 2 (2^j k + \varepsilon)$, so
    \begin{equation*}
        y^2 - 1 = (y - \varepsilon)(y + \varepsilon) = 2^{j+2} k \bigl( 2^j k + \varepsilon \bigr) ,
    \end{equation*}
    which gives $S(y)$. The number $2^j k + \varepsilon$ is odd, since $j \Ge 1$, so $j$ halvings take $S(y)$ to the odd number $y' = k (2^j k + \varepsilon)$ through multiples of $y'$. If $k \Ge 3$, then
    \begin{equation*}
        y' - y = 2^j k (k - 2) + \varepsilon (k - 1) \Ge 2k (k - 2) - (k - 1) > 0 .
    \end{equation*}
\end{proof}

With $k$ as in the decomposition \eqref{eq: decomposition}, an odd number $y \Ge 3$ is a terminal number of the map (Table~\ref{tab: maps}) if and only if $k = 1$. We next describe the orbit of a terminal number.

Under $Q_-$ and $Q_+$, let $y = 2^m + \varepsilon \Ge 3$ be a terminal number. Lemma~\ref{lem: odd step} with $j = m$ and $k = 1$ gives $Q_-^m(y) = y$ under $Q_-$ and $Q_+^m(y) = y$ under $Q_+$. The orbit of $y$ runs through the $m$ distinct numbers $2^i (2^m + \varepsilon)$ with $0 \Le i \Le m - 1$, which form the cycle of length $m$
\begin{equation}
    2^m + \varepsilon \to 2^{m-1} (2^m + \varepsilon) \to \cdots \to 2 (2^m + \varepsilon) \to 2^m + \varepsilon .
    \label{eq: cycle}
\end{equation}
For $m = 1$, this cycle is the fixed point $3$ of $Q_-$. The terminal number $1$ of $Q_+$ is also a fixed point, since $Q_+(1) = 1$.

Under $S$, let $y = 2^{j+1} + \varepsilon$ be a terminal number. Lemma~\ref{lem: odd step} with $k = 1$ gives $y' = 2^j + \varepsilon$, which is again a terminal number. By repeated use of this step, the odd terms finally reach $1$ or $3$. Since $S$ maps $3 \to 2 \to 1 \to 0$, the orbit of every terminal number reaches the fixed point $0$. This is \citep[Lemma~6]{sedaghat2020bounded}.

The next lemma combines Lemma~\ref{lem: odd step} with the orbits of the terminal numbers, and describes every orbit whose initial value has an odd part greater than $1$.

\begin{lemma}
\label{lem: orbit}
    Let $x_0 \Ge 1$ have odd part $q > 1$, and let $(x_i)_{i \Ge 0}$ be its orbit under $Q_-$, $Q_+$ or $S$.
    \begin{enumerate}[(i)]
        \setlength{\itemsep}{0pt}
        \item Under $Q_-$ and $Q_+$, $q$ divides $x_i$ for every $i \Ge 0$.
        \item Let $y_0 = q$, $y_1$, $y_2$, $\dots$ be the odd terms of the orbit, in order of occurrence. For $y_r \Ge 3$, let $j_r$, $k_r$ and $\varepsilon_r$ be the numbers of the decomposition \eqref{eq: decomposition} for $y = y_r$. By Lemma~\ref{lem: odd step},
        \begin{equation}
            y_{r+1} = k_r \bigl( 2^{j_r} k_r + \varepsilon_r \bigr) .
            \label{eq: odd terms}
        \end{equation}
        If some $y_r$ is a terminal number, then under $Q_-$ and $Q_+$ the orbit enters the cycle \eqref{eq: cycle} that contains $y_r$, and under $S$ it reaches the fixed point $0$. Otherwise $k_r \Ge 3$ for every $r$, and $x_i \to \infty$ as $i \to \infty$.
    \end{enumerate}
\end{lemma}

\begin{proof}
    Under $Q_-$ and $Q_+$, let $q$ divide $n$. If $n$ is even, $q$ divides $n/2$, since $q$ is odd. If $n$ is odd, $Q_-(n) = n \cdot (n - 1)/2$ and $Q_+(n) = n \cdot (n + 1)/2$ are multiples of $n$. Hence $q$ divides $x_{i+1}$ whenever it divides $x_i$. Since $q$ divides $x_0$, induction on $i$ proves (i).

    The $l$ halvings take $x_0$ to $q = y_0$. If some $y_r$ is a terminal number, the claim follows from the discussion before this lemma, since (i) gives $y_r \Ge q \Ge 3$ under $Q_+$. Assume now that no $y_r$ is a terminal number. If $y_r \Ge 3$, then $k_r \Ge 3$, and Lemma~\ref{lem: odd step} gives $y_{r+1} > y_r$. Since $y_0 = q \Ge 3$, induction on $r$ gives $k_r \Ge 3$ and $y_{r+1} > y_r$ for every $r$. Hence $y_r \Ge q + r$. By Lemma~\ref{lem: odd step}, every term after $y_r$ is at least $y_{r+1}$, so $x_i \to \infty$. This proves (ii).
\end{proof}

For $Q_-$, Lemma~\ref{lem: orbit}~(ii) combines \citep[Theorem~2]{sedaghat2020bounded} with an observation in the proof of \citep[Theorem~3]{sedaghat2020dual}: an orbit that does not enter a cycle tends to infinity. For $S$, it contains \citep[Theorem~7]{sedaghat2020bounded}.

By Lemma~\ref{lem: orbit}~(ii), it remains to determine the initial values whose orbits have an odd term that is a terminal number.

\begin{lemma}
\label{lem: reduction}
    Let $x_0 \Ge 1$ have odd part $q > 1$, and consider its orbit under $Q_-$, $Q_+$ or $S$. Assume that $q$ is not a terminal number of the map, and that some odd term of the orbit is a terminal number. Let $r$ be the least index such that $y_r$ is a terminal number, where $y_r$ is as in Lemma~\ref{lem: orbit}~(ii). Then $r \Ge 1$, $k_{r-1} \Ge 3$, and $j = j_{r-1}$, $k = k_{r-1}$ give the following solutions of \eqref{eq: key}.
    \begin{enumerate}[(i)]
        \setlength{\itemsep}{0pt}
        \item Under $Q_-$, $y_r = 2^m + 1$ with $m \Ge 1$, and $k_{r-1} (2^{j_{r-1}} k_{r-1} + 1) = 2^m + 1$.
        \item Under $Q_+$, $y_r = 2^m - 1$ with $m \Ge 1$, and $k_{r-1} (2^{j_{r-1}} k_{r-1} - 1) = 2^m - 1$.
        \item Under $S$, $y_r = 2^m + \delta$ with $m \Ge 1$ and $\delta \in \{-1, 1\}$, and $k_{r-1} (2^{j_{r-1}} k_{r-1} + \varepsilon_{r-1}) = 2^m + \delta$.
    \end{enumerate}
\end{lemma}

\begin{proof}
    Since $y_0 = q$ is not a terminal number, we have $r \Ge 1$. By the minimality of $r$, the terms $y_0, \dots, y_{r-1}$ are not terminal numbers. As in the proof of Lemma~\ref{lem: orbit}, induction gives $y_{r-1} \Ge 3$ and $k_{r-1} \Ge 3$. By \eqref{eq: odd terms},
    \begin{equation*}
        y_r = k_{r-1} \bigl( 2^{j_{r-1}} k_{r-1} + \varepsilon_{r-1} \bigr) ,
    \end{equation*}
    where $\varepsilon_{r-1} = 1$ under $Q_-$ and $\varepsilon_{r-1} = -1$ under $Q_+$ by the decomposition \eqref{eq: decomposition}. The terminal number $y_r$ is $2^m + 1$ under $Q_-$, $2^m - 1$ under $Q_+$, and $2^m \pm 1$ under $S$, with $m \Ge 1$. This proves (i), (ii) and (iii).
\end{proof}

Section~\ref{sec: proofs} combines Lemma~\ref{lem: reduction} with the solutions of \eqref{eq: key} in Lemma~\ref{lem: key} to prove Theorems~\ref{thm: minus}, \ref{thm: plus} and~\ref{thm: S}.

\section{The exponential Diophantine equation}
\label{sec: equation}

We solve \eqref{eq: key} with the following result of \citet{szalay2002equations}, in the form stated by \citet[Theorems~1.1 and~1.2]{scott2006elementary}.

\begin{lemma}[\citet{szalay2002equations}]
\label{lem: szalay}
    Consider the solutions of $2^a \pm 2^b + 1 = z^2$ in positive integers $a$, $b$ and $z$.
    \begin{enumerate}[(i)]
        \setlength{\itemsep}{0pt}
        \item The solutions of $2^a + 2^b + 1 = z^2$ with $a \Ge b$ are
        \begin{equation*}
            (a, b, z) = (2u, u + 1, 2^u + 1) \ \text{ with } u \Ge 1 , \qquad (5, 4, 7) , \qquad (9, 4, 23) .
        \end{equation*}
        \item The solutions of $2^a - 2^b + 1 = z^2$ with $a > b$ are
        \begin{equation*}
            (a, b, z) = (2u, u + 1, 2^u - 1) \ \text{ with } u \Ge 2 , \qquad (5, 3, 5) , \qquad (7, 3, 11) , \qquad (15, 3, 181) .
        \end{equation*}
    \end{enumerate}
\end{lemma}

Apart from textbook results, Lemma~\ref{lem: szalay} is the only external result used in our proofs. By the identity \eqref{eq: square}, every solution of \eqref{eq: key} gives a square of the form $2^a \pm 2^b + 1$, so Lemma~\ref{lem: szalay} determines all solutions of \eqref{eq: key}.

\begin{lemma}
\label{lem: key}
    The solutions of \eqref{eq: key} are given by the identities \eqref{eq: solutions}, that is, $(j, k, \varepsilon, m, \delta)$ is $(1, 3, -1, 4, -1)$, $(2, 3, -1, 5, 1)$ or $(1, 45, 1, 12, -1)$. In particular, the following hold for \eqref{eq: key} with $\varepsilon = \delta$.
    \begin{enumerate}[(i)]
        \setlength{\itemsep}{0pt}
        \item The equation $k (2^j k + 1) = 2^m + 1$ has no solution.
        \item The only solution of $k (2^j k - 1) = 2^m - 1$ is $(j, k, m) = (1, 3, 4)$.
    \end{enumerate}
\end{lemma}

\begin{proof}
    The identities \eqref{eq: solutions} hold. Conversely, let $(j, k, \varepsilon, m, \delta)$ satisfy \eqref{eq: key}. By \eqref{eq: square},
    \begin{equation*}
        (a, b, z) = \bigl( m + j + 2 ,\ j + 2 ,\ 2^{j+1} k + \varepsilon \bigr)
    \end{equation*}
    solves the equation of Lemma~\ref{lem: szalay}~(i) if $\delta = 1$, and that of Lemma~\ref{lem: szalay}~(ii) if $\delta = -1$. Since $m \Ge 1$, the exponent $a$ exceeds $b$. Since $j \Ge 1$, we have $z \equiv \varepsilon \pmod 4$, so $z$ determines $\varepsilon$. Then $j = b - 2$, $k = (z - \varepsilon)/2^{j+1}$ and $m = a - j - 2$.

    For $\delta = 1$, Lemma~\ref{lem: szalay}~(i) gives $(a, b, z) = (2u, u + 1, 2^u + 1)$ with $u \Ge 1$, $(a, b, z) = (5, 4, 7)$ or $(a, b, z) = (9, 4, 23)$. In the family, $u = b - 1 = j + 1$, and $2^{j+1} k + \varepsilon = 2^{j+1} + 1$ gives $2^{j+1} (k - 1) = 1 - \varepsilon$. This is impossible, since $2^{j+1} (k - 1) \Ge 8$. The solution $(5, 4, 7)$ gives $j = 2$, $\varepsilon = -1$ and $k = 1$, which is excluded. The solution $(9, 4, 23)$ gives $(j, k, \varepsilon, m) = (2, 3, -1, 5)$.

    For $\delta = -1$, Lemma~\ref{lem: szalay}~(ii) gives $(a, b, z) = (2u, u + 1, 2^u - 1)$ with $u \Ge 2$, $(a, b, z) = (5, 3, 5)$, $(a, b, z) = (7, 3, 11)$ or $(a, b, z) = (15, 3, 181)$. In the family, $u = b - 1 = j + 1$, and $2^{j+1} k + \varepsilon = 2^{j+1} - 1$ gives $2^{j+1} (k - 1) = -1 - \varepsilon$. This is impossible, since $2^{j+1} (k - 1) \Ge 8$. The solution $(5, 3, 5)$ gives $j = 1$, $\varepsilon = 1$ and $k = 1$, which is excluded. The solutions $(7, 3, 11)$ and $(15, 3, 181)$ give $(j, k, \varepsilon, m) = (1, 3, -1, 4)$ and $(1, 45, 1, 12)$.

    These are the identities \eqref{eq: solutions}. None of them has $\varepsilon = \delta = 1$, and only $(j, k, \varepsilon, m, \delta) = (1, 3, -1, 4, -1)$ has $\varepsilon = \delta = -1$. This proves (i) and (ii).
\end{proof}

Lemma~\ref{lem: key}~(i) is the statement of \citep[Lemma~2]{sedaghat2020dual}. The next remark discusses the descent by which it is proved there.

\begin{remark}
\label{rem: descent}
    The proof of \citep[Lemma~2]{sedaghat2020dual} writes $k - 1 = 2^{\alpha_1} h_1$, then $h_1 + 1 = 2^{\alpha_2} h_2$. In general, it writes $h_i + 1$ or $h_i - 1$ as $2^{\alpha_{i+1}} h_{i+1}$ with $h_{i+1}$ odd. \citet{sedaghat2020dual} denotes $\alpha_i$ and $h_i$ by $j_i$ and $k_i$. At each reduction, the proof divides the equation by the least power of $2$ among its terms. The reduction is impossible when $h_i - 1 = 0$, and the proof does not treat this case.

    For fixed $j, m \Ge 1$, reducing $k (2^j k + 1) = 2^m + 1$ modulo any given power of $2$ does not exclude a solution. For fixed $j \Ge 1$ and integers $k$, $k'$,
    \begin{equation*}
        \bigl( 2^j k^2 + k \bigr) - \bigl( 2^j k'^2 + k' \bigr) = (k - k') \bigl( 2^j (k + k') + 1 \bigr) ,
    \end{equation*}
    and the factor $2^j (k + k') + 1$ is odd. Hence the map $k \mapsto 2^j k^2 + k$ is injective, and therefore bijective, modulo every power of $2$. In particular, the congruence $2^j k^2 + k \equiv 2^m + 1$ has a solution modulo every power of $2$, and this solution is odd, since $2^j k^2 + k \equiv k \pmod 2$.

    We apply the same descent to $k (2^j k - 1) = 2^m - 1$, which is \eqref{eq: key} with $\varepsilon = \delta = -1$. Write it as $2^j k^2 - k + 1 = 2^m$. Since $k \Ge 3$, this gives $2^m > 2^j$. Write $k - 1 = 2^{\alpha_1} h_1$ with $h_1$ odd, so that
    \begin{equation*}
        2^j k^2 - 2^{\alpha_1} h_1 = 2^m .
    \end{equation*}
    As $k$ and $h_1$ are odd and $m > j$, comparing the exponents of $2$ in the terms of this equation forces $\alpha_1 = j$. Division by $2^j$ gives $k^2 - h_1 = 2^{m-j}$. Substituting $k = 2^j h_1 + 1$ gives
    \begin{equation*}
        2^{2j} h_1^2 + 2^{j+1} h_1 + 1 - h_1 = 2^{m-j} ,
    \end{equation*}
    and the next reduction would write $h_1 - 1$ as $2^{\alpha_2} h_2$. At the solution $(j, k, m) = (1, 3, 4)$ of Lemma~\ref{lem: key}~(ii), $h_1 = 1$, and the descent stops. If the descent continued indefinitely, it would exclude this solution, and hence the bounded orbits of $5$ under $Q_+$ and of $11$ under $S$.
\end{remark}

\section{Proofs of the classification theorems}
\label{sec: proofs}

We now prove Theorems~\ref{thm: minus}, \ref{thm: plus} and~\ref{thm: S}. Each proof first treats the initial values with a bounded orbit, and then shows that every other orbit tends to infinity. To show this, we assume that some odd term of the orbit is a terminal number. Lemma~\ref{lem: reduction} then gives a solution of \eqref{eq: key}, and Lemma~\ref{lem: key} lists all solutions. Under $Q_-$, no solution exists. Under $Q_+$ and $S$, the solutions force the odd part of $x_0$ to be $5$ and to lie in $\{11, 23, 181\}$, respectively.

\begin{proof}[Proof of Theorem~\ref{thm: minus}]
    Since $Q_-(0) = Q_-(1) = 0$, the orbits of $0$ and of the powers of $2$ reach $0$, which gives (i). For $x_0 = 2^l (2^m + 1)$, the odd part $2^m + 1$ is a terminal number, and Lemma~\ref{lem: orbit}~(ii) gives (ii). Every other $x_0$ has an odd part $q > 1$ that is not a terminal number. If some odd term were a terminal number, Lemma~\ref{lem: reduction}~(i) would give a solution of \eqref{eq: key} with $\varepsilon = \delta = 1$, which contradicts Lemma~\ref{lem: key}~(i). Hence $x_i \to \infty$ by Lemma~\ref{lem: orbit}~(ii). This proves (iii).
\end{proof}

\begin{proof}[Proof of Theorem~\ref{thm: plus}]
    Part (i) holds since $Q_+(0) = 0$. For $x_0 = 2^l (2^m - 1)$ with $m \Ge 2$, the odd part $2^m - 1$ is a terminal number, and Lemma~\ref{lem: orbit}~(ii) applies. For $m = 1$, the orbit of $x_0 = 2^l$ reaches the fixed point $1$. This gives (ii). For $x_0 = 5 \cdot 2^l$, the odd part $y = 5 = 2 \cdot 3 - 1$ has $j = 1$ and $k = 3$ in the decomposition \eqref{eq: decomposition}, and Lemma~\ref{lem: odd step} gives the terminal number $y' = 15 = 2^4 - 1$. Lemma~\ref{lem: orbit}~(ii) gives (iii).

    Every other $x_0$ has an odd part $q \Ge 3$ that is neither a terminal number nor $5$. Assume that some odd term is a terminal number, and let $r$ and $y_r = 2^m - 1$ be as in Lemma~\ref{lem: reduction}~(ii). By Lemma~\ref{lem: key}~(ii), $(j_{r-1}, k_{r-1}, m) = (1, 3, 4)$, so $y_{r-1} = 2 \cdot 3 - 1 = 5$ by the decomposition \eqref{eq: decomposition}. By Lemma~\ref{lem: orbit}~(i), $q$ divides $y_{r-1} = 5$. Since $q \Ge 3$, this gives $q = 5$, which is excluded. Hence no odd term is a terminal number, and $x_i \to \infty$ by Lemma~\ref{lem: orbit}~(ii). This proves (iv).
\end{proof}

\begin{proof}[Proof of Theorem~\ref{thm: S}]
    The orbit of $0$ stays at $0$. If the odd part $q$ of $x_0$ is a terminal number, the orbit reaches $0$ by the discussion before Lemma~\ref{lem: orbit}. If $y = q$ is $11 = 4 \cdot 3 - 1$, $23 = 8 \cdot 3 - 1$ or $181 = 4 \cdot 45 + 1$, written as in the decomposition \eqref{eq: decomposition}, Lemma~\ref{lem: odd step} gives $y' = 15 = 2^4 - 1$, $33 = 2^5 + 1$ or $4095 = 2^{12} - 1$, respectively. Then $y'$ is a terminal number, and the orbit reaches $0$ by Lemma~\ref{lem: orbit}~(ii). This proves (i).

    For (ii), let $x_0 \Ge 1$ have an odd part $q$ that is neither a terminal number nor in $\{11, 23, 181\}$. Then $q > 1$, since $1 = 2^1 - 1$ is a terminal number. Assume that some odd term is a terminal number, and let $r$ be as in Lemma~\ref{lem: reduction}~(iii). By Lemma~\ref{lem: key} and the decomposition \eqref{eq: decomposition}, the term $y_{r-1} = 2^{j_{r-1} + 1} k_{r-1} + \varepsilon_{r-1}$ is $11$, $23$ or $181$, and each of these numbers is prime. If $r \Ge 2$, then \eqref{eq: odd terms} gives $y_{r-1} = k_{r-2} (2^{j_{r-2}} k_{r-2} + \varepsilon_{r-2})$. As in the proof of Lemma~\ref{lem: reduction}, $k_{r-2} \Ge 3$, so $2^{j_{r-2}} k_{r-2} + \varepsilon_{r-2} \Ge 5$ and $y_{r-1}$ is composite. Hence $r = 1$ and $q = y_0$ lies in $\{11, 23, 181\}$, a contradiction. So no odd term is a terminal number, and Lemma~\ref{lem: orbit}~(ii) proves (ii).
\end{proof}

Theorems~\ref{thm: minus}, \ref{thm: plus} and~\ref{thm: S} have the following consequences.

\begin{corollary}
\label{cor: bounded}
    \begin{enumerate}[(i)]
        \setlength{\itemsep}{0pt}
        \item For $X \Ge 2$, at most $(\log_2 X + 1)^2$ initial values in $\{0, 1, \dots, X\}$ have a bounded orbit under $Q_-$, and the same bound holds under $Q_+$.
        \item For $X \Ge 2$, at most $(\log_2 X + 3)^2$ initial values in $\{0, 1, \dots, X\}$ have a bounded orbit under $S$.
        \item An orbit under $Q_-$ that contains an odd term $2^m - 1$ with $m \Ge 3$ tends to infinity.
        \item An orbit under $Q_+$ that contains an odd term $2^m + 1$ with $m \Ge 3$ tends to infinity.
    \end{enumerate}
\end{corollary}

\begin{proof}
    Under $Q_-$, by Theorem~\ref{thm: minus}, a bounded orbit starts from $0$, from a power $2^l \Le X$, or from $2^l (2^m + 1) \Le X$. There are at most $\log_2 X + 1$ powers $2^l \Le X$. For the initial values $2^l (2^m + 1) \Le X$,
    \begin{equation*}
        2^{l+m} < 2^l (2^m + 1) \Le X , \qquad \text{so} \qquad l + m < \log_2 X ,
    \end{equation*}
    and at most $\log_2 X \, (\log_2 X + 1)/2$ pairs $(l, m)$ with $l \Ge 0$ and $m \Ge 1$ qualify. Since $\log_2 X \Ge 1$, the total is at most
    \begin{equation*}
        1 + (\log_2 X + 1) + \frac{\log_2 X \, (\log_2 X + 1)}{2} \Le (\log_2 X + 1)^2 .
    \end{equation*}

    Under $Q_+$, by Theorem~\ref{thm: plus}, a bounded orbit starts from $0$, from $2^l (2^m - 1) \Le X$, or from $5 \cdot 2^l \Le X$. For the initial values $2^l (2^m - 1) \Le X$,
    \begin{equation*}
        2^{l+m-1} \Le 2^l (2^m - 1) \Le X , \qquad \text{so} \qquad l + m \Le \log_2 X + 1 ,
    \end{equation*}
    and at most $(\log_2 X + 1)(\log_2 X + 2)/2$ pairs $(l, m)$ qualify. For $5 \cdot 2^l \Le X$, we have $2^{l+2} < X$, which leaves at most $\log_2 X - 1$ values of $l$. Since $\log_2 X \Ge 1$, the total is at most
    \begin{equation*}
        1 + \frac{(\log_2 X + 1)(\log_2 X + 2)}{2} + \log_2 X - 1 \Le (\log_2 X + 1)^2 .
    \end{equation*}
    This proves (i).

    Under $S$, by Theorem~\ref{thm: S}, a bounded orbit starts from $0$, from $2^l (2^m - 1) \Le X$, from $2^l (2^m + 1) \Le X$, or from $2^l c \Le X$ with $c \in \{11, 23, 181\}$. The counts above bound the families $2^l (2^m - 1)$ and $2^l (2^m + 1)$ by $(\log_2 X + 1)(\log_2 X + 2)/2$ and $\log_2 X \, (\log_2 X + 1)/2$, respectively. For each $c$, we have $2^{l+3} < 2^l c \Le X$, which leaves at most $\log_2 X$ values of $l$. The total is at most
    \begin{equation*}
        1 + \frac{(\log_2 X + 1)(\log_2 X + 2)}{2} + \frac{\log_2 X \, (\log_2 X + 1)}{2} + 3 \log_2 X \Le (\log_2 X + 3)^2 .
    \end{equation*}
    This proves (ii).

    Under $Q_-$, from the odd term $2^m - 1$ on, the orbit is the orbit of the initial value $2^m - 1$. For $m \Ge 3$, this number is at least $7$ and congruent to $3$ modulo $4$, whereas every terminal number of $Q_-$ that is at least $7$ is congruent to $1$ modulo $4$. So $2^m - 1$ is neither a terminal number nor a power of $2$, and Theorem~\ref{thm: minus}~(iii) applies. This proves (iii).

    Under $Q_+$, from the odd term $2^m + 1$ on, the orbit is the orbit of the initial value $2^m + 1$. For $m \Ge 3$, this number is at least $9$ and congruent to $1$ modulo $4$, whereas every terminal number of $Q_+$ that is at least $9$ is congruent to $3$ modulo $4$. So $2^m + 1$ is neither a terminal number nor $5$, and Theorem~\ref{thm: plus}~(iv) applies. This proves (iv).
\end{proof}

Corollary~\ref{cor: bounded}~(iii) proves \citep[Conjecture~4]{sedaghat2020bounded}. In Corollary~\ref{cor: bounded}~(iii) and~(iv), the condition $m \Ge 3$ cannot be relaxed. Under $Q_-$, the orbit of $2^1 - 1 = 1$ reaches $0$, and $2^2 - 1 = 3$ is a fixed point. Under $Q_+$, the number $2^1 + 1 = 3 = 2^2 - 1$ lies on the cycle $3 \to 6 \to 3$, and the orbit of $2^2 + 1 = 5$ is bounded by Theorem~\ref{thm: plus}~(iii). Under $S$, every number $2^m \pm 1$ is a terminal number, so Corollary~\ref{cor: bounded}~(iii) and~(iv) have no counterpart for $S$.

\section{A divisibility criterion for \texorpdfstring{$Q_-$}{Q-}}
\label{sec: criterion}

The proof of Theorem~\ref{thm: minus} relies on Lemma~\ref{lem: szalay}. In this section, we give a criterion for unbounded orbits under $Q_-$ that uses only divisibility. For an odd integer $n > 1$, let $\ord_n(2)$ be the multiplicative order of $2$ modulo $n$, the least $d \Ge 1$ with $2^d \equiv 1 \pmod n$. For an integer $d \Ge 1$, let $v_2(d)$ be its $2$-adic valuation, the exponent of the largest power of $2$ dividing $d$.

\begin{theorem}
\label{thm: criterion}
    \begin{enumerate}[(i)]
        \setlength{\itemsep}{0pt}
        \item Let $x_0 \Ge 1$, and let $y > 1$ be an odd term of its orbit $(x_i)_{i \Ge 0}$ under $Q_-$. If $y$ divides none of the numbers $2^m + 1$ with $m \Ge 1$, then $x_i \to \infty$ as $i \to \infty$.
        \item An odd integer $y > 1$ divides $2^m + 1$ for some $m \Ge 1$ if and only if there is $s \Ge 1$ with $v_2(\ord_p(2)) = s$ for every prime $p$ dividing $y$.
    \end{enumerate}
\end{theorem}

\begin{proof}
    Let $i_0$ be an index with $x_{i_0} = y$. Lemma~\ref{lem: orbit}~(i), applied to the orbit of $y$, shows that $y$ divides $x_i$ for every $i \Ge i_0$. If $x_i$ does not tend to infinity, Lemma~\ref{lem: orbit}~(ii), applied to the orbit of $y$, gives an index $i_1 \Ge i_0$ and an integer $m \Ge 1$ with $x_{i_1} = 2^m + 1$. Then $y$ divides $2^m + 1$, which contradicts the assumption on $y$. This proves (i).

    For (ii), let $p$ be a prime dividing $y$, and let $p^e$ be the largest power of $p$ dividing $y$. We show that, for every $m \Ge 1$,
    \begin{equation*}
        2^m \equiv -1 \pmod{p^e} \quad \Longleftrightarrow \quad \ord_{p^e}(2) \text{ divides } 2m \text{ and not } m .
    \end{equation*}
    If $2^m \equiv -1 \pmod{p^e}$, then $2^{2m} \equiv 1$ and $2^m \not\equiv 1$ modulo $p^e$. This gives the forward implication. Conversely, assume that $\ord_{p^e}(2)$ divides $2m$ and not $m$. Then $p^e$ divides the product $(2^m - 1)(2^m + 1)$ but not $2^m - 1$. These factors differ by $2$, so $p$ does not divide both, and $p^e$ divides $2^m + 1$. Equivalently, the condition on $\ord_{p^e}(2)$ reads
    \begin{equation}
        v_2(m) = v_2 \bigl( \ord_{p^e}(2) \bigr) - 1 \quad \text{and} \quad \text{the odd part of } \ord_{p^e}(2) \text{ divides } m .
        \label{eq: local condition}
    \end{equation}
    Moreover, $v_2(\ord_{p^e}(2)) = v_2(\ord_p(2))$. Indeed, $\ord_p(2)$ divides $\ord_{p^e}(2)$, since $2^{\ord_{p^e}(2)} \equiv 1 \pmod p$. By the binomial theorem, if an integer is congruent to $1$ modulo $p^f$ with $f \Ge 1$, its $p$-th power is congruent to $1$ modulo $p^{f+1}$. Applying this $e - 1$ times, starting from $2^{\ord_p(2)} \equiv 1 \pmod p$, gives
    \begin{equation*}
        2^{\ord_p(2) \, p^{e-1}} \equiv 1 \pmod{p^e} ,
    \end{equation*}
    so $\ord_{p^e}(2)$ divides $\ord_p(2) \, p^{e-1}$. The quotient $\ord_{p^e}(2) / \ord_p(2)$ therefore divides the odd number $p^{e-1}$, and the $2$-adic valuations agree.

    Now $y$ divides $2^m + 1$ if and only if \eqref{eq: local condition} holds for every prime $p$ dividing $y$. If such an $m$ exists, then $v_2(\ord_p(2)) = v_2(m) + 1$ for every prime $p$ dividing $y$. Conversely, assume that $v_2(\ord_p(2)) = s \Ge 1$ for every such $p$. Let $m$ be $2^{s-1}$ times the least common multiple of the odd parts of the orders $\ord_{p^e}(2)$. Then $m$ satisfies \eqref{eq: local condition} for every $p$. This proves (ii).
\end{proof}

We now apply Theorem~\ref{thm: criterion} to the conjectures of \citet{sedaghat2020bounded}. For $m \Ge 3$, the residues of the powers of $2$ modulo $2^m - 1$ are $2^0, 2^1, \dots, 2^{m-1}$, and none of them is congruent to $-1$, since $2^{m-1} < 2^m - 2$. Hence $2^m - 1$ divides no number $2^{m'} + 1$, and Theorem~\ref{thm: criterion}~(i) applies to every odd term $2^m - 1$. This proves Corollary~\ref{cor: bounded}~(iii) without Lemma~\ref{lem: szalay}. With $x_0 = 7 = 2^3 - 1$, Theorem~\ref{thm: criterion}~(i) also proves that unbounded orbits exist, which is \citep[Conjecture~3]{sedaghat2020bounded}. In \citep{sedaghat2020dual}, both conjectures are derived from the classification stated there.

Without Lemma~\ref{lem: szalay}, Theorem~\ref{thm: criterion} also shows that the initial values with a bounded orbit under $Q_-$ have natural density $0$, that is, only $o(X)$ of them lie in $\{0, 1, \dots, X\}$. This was announced in \citep{sedaghat2020dual}. Let $p \equiv 7 \pmod 8$ be prime. By Euler's criterion and the supplementary law of quadratic reciprocity for $2$ \citep[Theorem~12.1~(i) and~(iv)]{shoup2008computational}, $2^{(p-1)/2} \equiv 1 \pmod p$. Since $(p-1)/2$ is odd, $v_2(\ord_p(2)) = 0$, and by Theorem~\ref{thm: criterion}~(ii) no odd multiple of $p$ divides a number $2^m + 1$. By Theorem~\ref{thm: criterion}~(i), applied to the odd part of $x_0$, every $x_0 \Ge 1$ with a prime factor $p \equiv 7 \pmod 8$ has an unbounded orbit.

For a finite set $P$ of primes $p \equiv 7 \pmod 8$, inclusion--exclusion shows that the number of $x_0 \in \{1, \dots, X\}$ divisible by no $p \in P$ is
\begin{equation}
    X \prod_{p \in P} \Bigl( 1 - \frac 1p \Bigr) + O \bigl( 2^{|P|} \bigr) .
    \label{eq: sieve}
\end{equation}
By the prime number theorem for arithmetic progressions \citep[Theorem~5.21]{shoup2008computational}, the number of primes $p \Le X$ with $p \equiv 7 \pmod 8$ is asymptotic to $X / (4 \log X)$ as $X \to \infty$, and partial summation shows that the sum of $1/p$ over these primes diverges. Hence the product $\prod_{p \in P} (1 - 1/p)$ in \eqref{eq: sieve} can be made arbitrarily small by enlarging $P$. Every $x_0 \Ge 1$ with a bounded orbit is divisible by no $p \in P$. So, for every such $P$, the number of these initial values in $\{1, \dots, X\}$ is at most \eqref{eq: sieve}, and therefore this number is $o(X)$. Corollary~\ref{cor: bounded}~(i) gives the much stronger bound $(\log_2 X + 1)^2$, but its proof uses Lemma~\ref{lem: szalay}.

The criterion does not carry over to $Q_+$ and $S$. Under $Q_+$, Lemma~\ref{lem: orbit} shows that a bounded orbit of $x_0 = 2^l q$ with $q > 1$ enters a cycle that contains a number $2^m - 1$ divisible by $q$. The analogue of Theorem~\ref{thm: criterion}~(i) for $Q_+$ would state that an orbit with an odd term $y > 1$ that divides none of the numbers $2^m - 1$ tends to infinity. But every odd $y > 1$ divides $2^{\ord_y(2)} - 1$, so this statement applies to no orbit. Under $S$, the numbers $(y-1)/2$ and $(y+1)/2$ are coprime to an odd $y > 1$, so $y$ does not divide $S(y)$, and Lemma~\ref{lem: orbit}~(i) has no analogue.

The proof of Theorem~\ref{thm: plus} uses Lemma~\ref{lem: key} only through (ii), which concerns $\delta = -1$ and hence needs only Lemma~\ref{lem: szalay}~(ii). According to \citet[Section~1]{scott2006elementary}, Szalay's proof of Lemma~\ref{lem: szalay}~(ii) uses a non-elementary bound of Beukers, and replacing this bound by an earlier elementary theorem of Beukers makes the proof elementary. Hence Theorem~\ref{thm: plus} has an elementary proof. The proof of Theorem~\ref{thm: S} uses both parts of Lemma~\ref{lem: szalay}.

\section{Conclusion}
\label{sec: conclusion}

We leave the following questions open.

Under $Q_-$, does some initial value not covered by Theorem~\ref{thm: minus}~(i) and~(ii) have an orbit in which every odd term divides some number $2^m + 1$? By Theorem~\ref{thm: criterion}~(i), an elementary proof that no such initial value exists would give an elementary proof of Theorem~\ref{thm: minus}~(iii). Is there an elementary proof that \eqref{eq: key} has no solution with $\varepsilon = \delta = 1$? Such a proof would make Theorem~\ref{thm: minus} elementary.

\section*{Data availability}

The results of this paper and Lemma~\ref{lem: szalay} are verified in Lean~4 with the Mathlib library. The verification takes the results of Beukers used by \citet{szalay2002equations} as axioms. The Lean code is available at \url{https://github.com/xuda-ye-math/Collatz}. The Lean code is written with the assistance of Claude (Opus 5.5).

\bibliographystyle{plainnat}
\bibliography{references}

\end{document}